\documentclass[10pt,oneside,reqno]{amsart}

\usepackage{amssymb}
\usepackage{amsthm} 
\usepackage{amsmath,amsthm,amsopn,amstext,amsbsy}      
\usepackage{latexsym}
\usepackage{newlfont}
\usepackage{amsfonts}

\newtheorem{teo}{Theorem}[section]

\newtheorem{prop}[teo]{Proposition}
\newtheorem{lem}[teo]{Lemma}

\newenvironment{dem}{\textsc{Proof:}}{\begin{flushright}$\square$\end{flushright}}
\title{Bouchaud Walks with variable drift}
\author{Manuel Cabezas Parra}
\date{\today}
\email{mncabeza@mat.puc.cl}
\address{ Facultad de Matem\'aticas\\
Pontificia Universidad Cat\'olica de Chile\\
Casilla 306-Correo 22, Santiago 6904411, Chile\\
Telephone: [56](2)354-1476\\
Telefax: [56](2)552-5916}

\begin{document}

\begin{abstract}
 We study the possible scaling limits of a sequence of Bouchaud trap models on $\mathbb{Z}$ with a drift which decays to $0$ as we rescale the walks: as the time parameter is rescaled by $n$, the drift decays as $n^{-a}$, for some fixed $a\geq0$. Depending on the speed of the decay of the drift  we obtain three different scaling limits. If the drift decays slowly as we rescale the walks (small $a$), we obtain the inverse of an $\alpha$-stable subordinator as scaling limit.
If the drift decays quickly as we rescale the walks (big $a$), we obtain the F.I.N. diffusion as scaling limit.
 There is a critical speed of decay separating these two main regimes, where a new process appears as scaling limit. This new process is a drifted Brownian motion with a random, purely atomic speed measure. The critical speed of decay $a_c$ of the drift is related to the index $\alpha$ of
the inhomogeneity of the environment by the equation $a_c=\alpha/(\alpha+1)$.
\end{abstract}
\setlength{\baselineskip}{6mm}
 \maketitle

\noindent {\footnotesize

}

\section{\textbf{Introduction}}

The \textit{Bouchaud trap model} (B.T.M.) is a continuous time random walk $X$ on a graph $\cal G$ with random jump rates. To each vertex $x$ of $\cal G$ we assign a positive number $\tau_x$
where $(\tau_x)_{x \in \cal{G}}$ is an i.i.d. sequence such that
\begin{equation}\label{heavytails}
\lim_{u\rightarrow\infty}u^\alpha\mathbb{P}[\tau_x\geq u]=1
\end{equation}
with $\alpha \in (0,1)$. This means that the distribution of $\tau_x$ has heavy tails.
Each visit of $X$ to $x\in \cal{G}$ lasts an exponentially distributed time with mean $\tau_x$.
Let $S(k)$ be the time of the $k$-th jump of $X$. $(S(k),k\in\mathbb{N})$ is called the \textit{clock process} of $X$.
Let $Y_k:=X(S(k))$ be the position of $X$ after the $k$-th jump. $(Y_k:k\in\mathbb{N})$ is called the \textit{embedded discrete time random walk} associated to $X$.

This model was introduced by J.-P. Bouchaud in \cite{weakergodicitybreaking} and has been studied by physicists as a toy model for the analysis of the dynamics of some complex systems such as spin-glasses.
More precisely, each vertex $x$ of $\cal{G}$ corresponds to a metastable state of the complex system, and $X$ represents the trajectory of the system over its phase space.
One of the phenomena that this model has  helped to understand is that of aging, a characteristic feature of the slow dynamics of many
metastable systems. For an account of the physical literature on the B.T.M. we refer to \cite{phys}.

The model has also been studied by mathematicians on different graphs, exhibiting a variety of behaviors. In \cite{fin}, Fontes, Isopi and Newman analyze the one-dimensional case ($\cal{G}=\mathbb{Z}$)and where the walk $X$ is symmetric.
They obtain a scaling limit for $X$ which is called  the \textit{Fontes-Isopi-Newman} (F.I.N.) singular diffusion. This diffusion is a speed measure change of a
Brownian motion by a random, purely atomic measure $\rho$, where $\rho$ is the Stieltjes measure associated to an $\alpha$-stable subordinator.
Different aging regimes for the one-dimensional case where found by Ben-Arous and C\v{e}rn\'{y} in \cite{bc}. In higher dimensions ($\cal G=\mathbb{Z}^d, d \geq 2$), the symmetric model has a behavior
completely different to the one-dimensional case, as shown by Ben Arous and C\v{e}rn\'{y} in \cite{zetade}, and by Ben Arous, C\v{e}rn\'{y} and Mountford in \cite{two}.
In these papers, a scaling limit and aging results where obtained for $X$.
The scaling limit is called \textit{fractional kinetic process} (F.K.P)
which is a time-change of a $d$-dimensional Brownian motion by the inverse of an $\alpha$-stable subordinator.
In \cite{bbg} and \cite{bbg2} Ben Arous, Bovier and Gayrard obtained aging properties of the model on the complete graph. A study of this walk for a wider class of graphs can be found
on \cite{universal}. For a general account on the mathematical study of the model, we refer to \cite{bcnotes}.

The difference between the one dimensional case and the model in higher dimensions can be understood as follows.
We can express the clock process $S(k)$ of $X$ as $S(k)=\sum_{i=0}^{k-1}\tau_{Y_i}e_i$, where the $e_i$ are standard i.i.d. exponential random variables. Thus, the increments of
$S(k)$ are the depths of the traps $(\tau_x)_{x\in\cal{G}}$ as sampled by $Y_k$.
In the model in dimensions higher than two, the embedded discrete time random walk $Y_k$ is transient (the case $d=2$ is more delicate). Thus $Y_k$ will sample each trap $\tau_x$ a finite number of times.
That implies that $S(k)$ does not have long range interactions with its past and its scaling limit will be a Markovian process, which is an $\alpha$-stable subordinator.
On the other hand, in the one-dimensional symmetric B.T.M., we have that the embedded discrete time random walk $Y_k$ is recurrent.
 Thus $Y_k$ will sample each trap $\tau_x$ an infinite number of times.
In this case, $S(k)$ has long range interactions with its past and its scaling limit will be non-Markovian.
Furthermore, the clock process $S(k)$ will converge to the local time of a Brownian motion integrated against the random measure $\rho$.
Here $\rho$ plays the role of a scaling limit for the environment $(\tau_x)_{x\in\mathbb{Z}}$.

It is natural to ask if we can find intermediate behaviors between the transient case $(d\geq1)$ and the recurrent case $(d=1)$: if we introduce a drift to the one-dimensional B.T.M., note that the embedded discrete random walk becomes transient. Thus, intermediate behaviors between the transient and the recurrent case might appear when one analyzes a sequence of one-dimensional B.T.M.'s with a drift that decreases to $0$ as we rescale the walks. In this paper we study this question, showing that the speed of decay of the drift sets the long-term behavior of the model and exhibiting a sharp phase transition in terms of the type of limiting processes obtained. We next describe with more precision the way in which we define the B.T.M. with drift and the results that are obtained in this paper.

 For each $\epsilon>0$, denote by $X^{\epsilon}$ the B.T.M. on $\mathbb{Z}$ where the transition probabilities of the embedded discrete time random walk are $\frac{1+\epsilon}{2}$ to the right and $\frac{1-\epsilon}{2}$ to the left. We will call this process the B.T.M. with drift $\epsilon$.
 For $a\geq0$, consider a rescaled sequence of B.T.M's with drift $n^{-a}$,
 $(h_a(n)X^{n^{-a}}(tn);t\geq0)$, indexed by $n$, where $h_a(n)$ is an
 appropriate space scaling depending on $a$.
We will see that as the drift decays slowly (small $a$), the sequence
of walks converges to the inverse of an $\alpha$-stable subordinator, whereas
if the drift decays fast (large $a$) the limiting process is the F.I.N.
diffusion. As these two posibilities are qualitatively different, we are led to think that there is either, a gradual interpolation between these two behaviors as the
speed of decay changes, or a sharp transition between them as the speed of decay changes. We establish that there is a sharp transition between the two
scaling limits, that there is a critical speed of decay where a new, previously, process appears and that the transition happens at $a=\alpha/(\alpha+1)$.
 As the main theorem of this paper, we prove that, depending on the value of
 $a$, there are three different scaling limits:

\medskip

\begin{itemize}

\item
 \textbf{Supercritical case} ($a<\alpha/(\alpha+1)$).
  The sequence of walks  converges to the inverse of an $\alpha$-stable subordinator.

\item
\textbf{Critical case} ($a=\alpha/(\alpha+1)$).
 The sequence of walks converges to a process which is a speed measure change of a Brownian motion with drift that we will call the \textit{drifted F.I.N. diffusion}.

\item
 \textbf{Subcritical case } ($a>\alpha/(\alpha+1)$).
The sequence of walks converges to the F.I.N. diffusion.

\end{itemize}
\medskip

\noindent The case $a=0$ (contained in the supercritical case), which corresponds to a constant drift, was already addressed by Zindy in \cite{zindy}.

Let us now make a few remarks concerning the proof of our main theorem.
 The strategy of the proof for the supercritical case is a generalization of the method
used in \cite{zindy} and relies on the analysis of the sequence of processes of first hitting times
$(H^n_b(x);x \in [0,nS])$  ($ S $ is fixed, $b>0$) defined as
\begin{equation}\label{hitting}
H^n_b(x):=\inf\{t:X^{n^{-b}}(t) \geq x\} .
\end{equation}
 We  show that these processes (properly rescaled) converge to an $\alpha$-stable
subordinator. From that, it follows that the maximum of the walks converges to the inverse of an $\alpha$-stable subordinator.
This part of the proof requires some care, because, as we are working with a sequence of walks with variable drift, we cannot apply directly the methods used in \cite{zindy}. It turns out that we have to choose $b$ properly to obtain a sequence of walks with the desired drift as we invert the hitting time processes.
Then, it is easy to pass from the maximum of the walk to the walk itself.
In \cite{fin}
The proof corresponding to the critical case follows the arguments used by
\cite{fin}. There they express rescaled, symmetric one-dimensional B.T.M.'s as speed measure changes of a Brownian motion trough a random speed measure. But here we are working with asymmetric walks, so we cannot work with the expression used there. To treat the asymmetry of the walks, we use a Brownian motion with drift instead of a Brownian motion. That is, we express each walk $X^{n^{-\alpha/(\alpha+1)}}$ as a speed measure change of a Brownian motion with drift, and then prove convergence of the sequence of speed measures to $\rho$.
The latter is achieved by means of a coupling of the environments.
 In the subcritical case, although we obtain the same scaling limit as in \cite{fin} (a F.I.N. diffusion), again, because of the asymmetry of the model, we cannot work with the expression used there.  We deal with this obstacle using, besides a random speed measure, a scaling function.
 That is, we express the rescaled walks as time-scale changes of a Brownian motion. Then we prove that the scale change can be
neglected and show convergence of the sequence of speed measures to the random measure $\rho$.

The organization of the paper is as follows. In section \ref{results} we give the definition of the model and state our main results. There we also give simple heuristic arguments to understand the transition at $a=\alpha/(\alpha+1)$.
In section \ref{super} we obtain the behavior for the supercritical
case, and in section \ref{critical} we obtain the scaling limit for the
 critical case.
The behavior for the subcritical case is obtained in section \ref{sub}.

Finally, we would like to mention that while preparing the
final version of this article we have learned that Theorem (\ref{main})
has been independently obtained by Gantert, M\"{o}rters and Wachtel
\cite{gmw}. There, they also obtain aging results for the
B.T.M. with vanishing drift.

\textbf{Acknowledgements}:
 The author was supported by a fellowship of  the National Commission on Science and Technology of Chile (Conicyt)\#29100243.
 This paper contains material presented at the probability seminar at the P.U.C. of Chile in September 2009.

\section{\textbf{Notations and Main Results}}\label{results}
A Bouchaud trap model on $\mathbb{Z}$ with drift $\epsilon$, $(X^\epsilon(t);t\in[0,\infty])$ is a homogeneous Markov process with jump rates:\\
\label{generator}
 \begin{equation} c (x,y) := \left\lbrace
           \begin{array}{c l}
(1+\epsilon)\tau_x^{-1}/2    \text{   if}\ y=x+1 \\
(1-\epsilon)\tau_x^{-1}/2    \text{   if}\ y=x-1 \\

           \end{array}
         \right., \end{equation}\\

\noindent
where $\tau=(\tau_x)_{x\in \mathbb Z}$ are positive, i.i.d. under a measure $P$ and satisfy\
\begin{equation}
\label{tail}\lim_{u\rightarrow\infty}u^{\alpha} P[\tau_x\geq u]=1.
\end{equation}

For any topological space $E$, $\mathcal{B}(E)$ will stand for the $\sigma$-algebra of Borelians of $E$.
$\mathbb{P}_{\tau}^x$ and $\mathbb{E}_{\tau}^x$ will denote the probability and expectation conditioned on the environment $\tau=(\tau_x)_{x\in\mathbb{Z}}$
and with $X^{\epsilon}(0)=x$.
 These probabilities are often referred as quenched probabilities.
We define $\mathbb{P}^x$ on $\mathbb{Z}^{\mathbb{N}}\times{\mathbb{R}^+}^{\mathbb{Z}}$ stating that
for every $A\in\mathcal B(\mathbb{Z}^{\mathbb{N}})$ and $B\in\mathcal B({\mathbb{R}^+}^{\mathbb{Z}})$,
$\mathbb{P}^x[A\times B]:=\int_B \mathbb{P}^x_{\tau}[C_{\tau}] P(d\tau),$
where $C_{\tau}:=\{x\in\mathbb{Z}^{\mathbb{N}}:(x,\tau)\in A\times B\}.$

$\mathbb{P}^x$ is called the annealed probability. Note that $X^{\epsilon}$ is Markovian w.r.t. $\mathbb{P}^x_{\tau}$ but non-Markovian w.r.t. $\mathbb{P}^x$.
$\mathbb{E}^x$ is the expectation associated to $\mathbb{P}^x$. $\mathbb{P}^0$ and $\mathbb{E}^0$ will be simply denoted as $\mathbb{P}$ and $\mathbb{E}$. Also $\mathbb{P}_{\tau}$
and $\mathbb{E}_{\tau}$ will stand for $\mathbb{P}_\tau^0$ and $\mathbb{E}_\tau^0$ respectively.
These notations will be used with the same meaning for all the processes appearing in this paper.

We have to make some definitions in order to state our main result:
let $B(t)$ be a standard one dimensional Brownian motion starting at zero and $l(t,x)$ be a bi-continuous version of his local time.
Given any locally finite measure $\mu$ on $\mathbb{R}$, denote
$$\phi_{\mu}(s):=\int_{\mathbb{R}}l(s,y)\mu(dy),$$
and its right continuous generalized inverse by
$$\psi_{\mu}(t):=inf\{s>0:\phi_{\mu}(s)>t \}.$$

The right continuous generalized inverse exists by definition, is increasing and, as its name indicates, it is a right continuous function.
Then we define the speed measure change of $B$ with speed measure $\mu$, $X(\mu)(t)$ as
\begin{equation}\label{timechange}
X(\mu)(t):=B(\psi_{\mu}(t)).
\end{equation}

 We also need to define speed measure changes of a drifted Brownian motion.
 Let $C(t):=B(t)+t$. We know that $C(t)$ has a bi-continuous local time $\tilde{l}(t,y)$.
 Given any locally finite measure $\mu$ in $\mathbb R$ we
 define
 $$\tilde{\phi}_{\mu}(s):=\int_{\mathbb R} \tilde{l}(s,y) \mu (dy),$$
 and its generalized right-continuous inverse by
 $$\tilde{\psi}_{\mu}(t):= \inf\{s>0:\tilde{\phi}_{\mu}(s)>t\}.$$
 Then we define $\tilde{X}(\mu)(t)$ (the speed measure change of $C$ with speed measure $\mu$) by
 \begin{equation}\label{timechange2}
\tilde{X}(\mu)(t):=C(\tilde{\psi}_{\mu}(t)).
\end{equation}

 By changing the starting point of our underlying Brownian motion $B$, we can change the starting point of $\tilde{X}(\mu)$ and $X(\mu)$.
  
Let $(x_i,v_i)$ be an inhomogeneous Poisson point process on
$\mathbb{R}\times\mathbb{R}^+$, independent of $B$ with intensity measure
$\alpha v^{-1-\alpha}dxdv$. We define the random measure $\rho$ as
\begin{equation}\label{rho}
 \rho:=\sum v_i \delta_{x_i}.
\end{equation}

The diffusion $(Z(t);t\in[0,T])$ defined as $Z(s):=B(\psi_{\rho}(s))$ is called the \textit{F.I.N diffusion}.
We also define the \textit{drifted F.I.N. diffusion} $\tilde{Z}(t)$ as
$\tilde{Z}(t):=C(\tilde{\psi}_{\rho}(t))$.

 $D[0,T]$ will denote the space of cadlag  functions from $[0,T]$ to $\mathbb{R}$. $(D[0,T],M_1)$, $(D[0,T],J_1)$ and $(D[0,T],U)$ will stand for $D[0,T]$ equipped with the
Skorohod-$M_1$, Skorohod-$J_1$, and uniform topology respectively. We refer to \cite{witt} for an account on these topologies.
 We define $(X^{(n,a)};t\in[0,T])$, a rescaling of a walk with drift $n^{-a}$, by
\begin{equation}
X^{(n,a)}(t):=\left\lbrace
\begin{array}{c l}
\frac{X^{n^{-a}}(tn)}{n^{\alpha(1-a)}}\text{if} \ a<\frac{\alpha}{1+\alpha}\\\\
\frac{X^{n^{-a}}(tn)}{n^{\alpha/(\alpha+1)}}\text{if}\ a\geq\frac{\alpha}{1+\alpha}\\

\end{array}
\right. .
\end{equation}

 Let $V_{\alpha}$ be an $\alpha$-stable subordinator started at zero. That is, $V_{\alpha}$ is the increasing Levy process with Laplace transform $\mathbb{E}[\exp(-\lambda V_{\alpha}(t))]=\exp(-t\lambda^{\alpha}).$
 Now we are in conditions to state the main result of this paper.
         \begin{teo}\label{main}
         For all $T>0$:

\begin{itemize}
         \item[(i)]If $a<\alpha/(\alpha+1)$ we have that
$(X^{(n,a)}(t);t\in [0,T])$
 converges in distribution to
          $(V_{\alpha}^{-1}(t); t\in (0,T))$ in $(D[0,T],U)$ where $V_{\alpha}^{-1}$
          is the right continuous generalized inverse of
          $V_{\alpha}$.
          \item[(ii)] If $a=\alpha/(\alpha+1)$ we have that $(X^{(n,a)}(t); t \in [0,T])$
converges in distribution
         to the drifted F.I.N. diffusion $(\tilde{Z}(t); t \in [0,T])$ on
         $(D[0,T],U)$.
         \item[(iii)] If $a>\alpha/(\alpha+1)$ we have that $(X^{(n,a)}(t); t \in [0;T])$ converges in distribution
         to the F.I.N. diffusion $(Z(t); t \in [0,t]) $ on $(D[0,T],U)$.
\end{itemize}

         \end{teo}
We present heuristic arguments to understand the transition at $a=\frac{\alpha}{1+\alpha}$. First we analyze a sequence of discrete time random walks.
Let $(S^{\epsilon}(i),i\in\mathbb{N})$ be a simple asymmetric random walk with drift $\epsilon$,
$S^{\epsilon}(i):=\sum_{k=1}^ib_k^{\epsilon},$
 where $(b_k^{\epsilon})_{i\in\mathbb{N}}$ is an i.i.d. sequence of random variables with:
$\mathbb{P}[b^{\epsilon}_k=1]=\frac{1+\epsilon}{2};\textrm{                   }\mathbb{P}[b^{\epsilon}_k=-1]=\frac{1-\epsilon}{2}.$
We want to find the possible scaling limits of $(S^{\epsilon(n)}(in);i \in [0,T])$, depending on the speed of decay of $\epsilon(n)$ to $0$ as $n\rightarrow\infty$.

We couple the sequence of walks $S^{\epsilon(n)}$ in the following way: Let $(U_i)_{i\in\mathbb{N}}$ be an i.i.d. sequence of uniformly distributed random variables taking values on
$[0,1]$. We require that $S^{\epsilon(n)}$ takes his $i$-th step to the right ($b_i^{\epsilon(n)}=1$) if $U_i>\frac{1-\epsilon(n)}{2}$ and to the left otherwise.
For each walk, we can decompose the steps into two groups: the first group
is given by the steps $i$ such that $\frac{1-\epsilon(n)}{2}<U_i<\frac{1+\epsilon(n)}{2}$ and the second group
consists of the remaining steps.
 We can think that the first group of steps takes account of the drift effect and the second one takes account of the symmetric fluctuations of the walk.

If the walk has given $n$ steps, then the first group has about $n\epsilon(n)$ steps, and the second group has fluctuations of order $\sqrt n$.
 It is obvious that the drift effect will dominate the behavior if $\sqrt n=o(\epsilon(n))$. In this case we will have a ballistic (deterministic) process as a scaling limit.
If $\epsilon(n)=o(\sqrt n)$ the fluctuations will dominate and we will have a Brownian motion as scaling limit.
Finally the two behaviors will be of the same order if $\epsilon(n)\approx\sqrt n$, and a Brownian motion with drift will be the scaling limit.

 The same reasoning can now be used to understand the change of behavior at $a=\alpha/(\alpha+1)$ for the sequence of walks $(X^{n^{-a}}(tn),t\in[0,T])_{n \in \mathbb N}$.
 In order to apply the precedent arguments we first have to estimate the number of steps that $X^{n^{-a}}$ has given up to time $Tn$. To simplify we take $T=1$. First, suppose that $X^{n^{-a}}(n)$ is of order $n^{u}$, where $u$ is to be found.
We know that after $k$ steps, a walk with drift $n^{-a}$ is approximately on site $kn^{-a}$, so, it takes about $n^{u+a}$ steps
to be on site $n^{u}$. Thus, we can also deduce that at time $n$, $X^{n^{-a}}$ has
visited approximately $n^{a}$ times each site.
As the distribution of $\tau_i$ satisfies (\ref{tail}), then the sum $\sum_{i=0}^{n^u}\tau_i$ is of the same order that $\max_{\{0\leq i \leq n^u\}}\tau_i$, and both are of order $n^{u/\alpha}$.
We can estimate the time needed to arrive at $n^{u}$ as the depth of the deepest trap found ($\approx n^{u/\alpha}$) multiplied by the number of visits to that trap ($\approx n^a$). This gives that $n\approx n^{\frac{u}{\alpha}+a}$.
But, we know, by definition, that $X^{n^{-a}}$ arrives at the site $n^{n/u}$ approximately at time $n$. It follows that $1=(u/\alpha)+a$, which yields $u=(1-a)\alpha$.
This means that the number of steps that $X^{n^{-a}}$ has given up to time $n$ is of order $n^{(1-a)\alpha+a}$.

 Again, we can decompose the steps of $X^{n^{-a}}$ into two groups. The first group accounts for the drift effect, and the second one accounts for the fluctuations.
 The first group will have approximately $n^{-a+[(1-a)\alpha+a]}$ steps and the second group will give a contribution to the position of order $n^{\frac{(1-a)\alpha+a}{2}}$.
Now it is easy to see that the ballistic behavior and the fluctuations will be of the same order i.f.f.
 $[(1-a)\alpha + a]/2=(1-a)\alpha$ or
 $a=\alpha/(1+\alpha).$
\section {\textbf{The Supercritical Regime}}\label{super}
The proof for the constant drift case ($a=0$) in \cite{zindy} is roughly as follows: first he prove that the sequence of rescaled first hitting times, $(n^{-1/\alpha}\inf\{s\geq0:X^{\epsilon}(ns)\geq x\}: x\geq 0)$, converges to an $\alpha$-stable subordinator. Then, using that the right continuous generalized inverse of the process of first hitting times is the maximum of $X^{\epsilon}(t)$, he can deduce that $(\max\{n^{-1}X^{\epsilon}(n^{1/\alpha}s): s\leq t\}: t\geq 0)$ converges to the inverse of an $\alpha$-stable subordinator. Finally he shows that the walk and its maximum are close.

For the proof of part (i) of theorem \ref{main} we cannot follow the proof of \cite{zindy} in a straightforward way:
suppose we show that a properly rescaled sequence of first hitting time processes $(p_a(n)H^n_a(nx):x\in\mathbb{R}_+)$ (where $p_a(n)$ is an appropriate scaling) converges to an $\alpha$-stable subordinator. Then, by inverting the processes, we get that the sequence $(\max\{n^{-1}X^{n^{-a}}(p_a(n)^{-1}s):s \leq t\}: t\in\mathbb{R}_+)$ converges to the inverse of an $\alpha$-stable subordinator. But we are searching a limit for $(\max\{d_a(n)X^{n^{-a}}(tn):t\in\mathbb{R}_+\})$ (where $d_a(n)$ is appropriate space scaling). That is, we want to obtain the limit of a sequence of rescaled walks where the drift decays as $n^{-a}$ when the time is rescaled by $n$. But when we invert $(p_a(n)H^n_a(nx):x\in\mathbb{R}_+)$, we obtain the sequence $(\max\{n^{-1}X^{n^{-a}}(p_a(n)^{-1}s):s\leq t\}: t\geq 0)$, which is a sequence of maximums of rescaled walks in which  the drift decays as $n^{-a}$ when the time is rescaled as $p_a(n)^{-1}$.

To solve this, we will prove that the limit of $(q_a(n)H^n_{b^{\ast}}(nx):x\in\mathbb{R}_+)$ is an $\alpha$-stable subordinator, where $q_a(n)$ is an appropriate scaling and $b^{\ast}$ sets an appropriate drift decay and depends on $a$.
Inverting, we will obtain that $(\max\{n^{-1}X^{n^{-b^{\ast}}}(q_a(n)^{-1}s):s\leq t\}:t\geq 0)$ converges to an $\alpha$-stable subordinator.
As we have said, we want the limit of a sequence of rescaled walks with
a drift that decays as $n^{-a}$ as the time parameter is rescaled by $n$.
Hence, when the time parameter is rescaled as $q_a(n)^{-1}$, the drift 
should rescale as $q_a(n)^{a}$. Thus we need to choose $b^{\ast}$ so that $n^{-b^{\ast}}=q_a(n)^{a}$. But we know that $q_a(n)$ is the appropriate scaling for $(H^n_{b^{\ast}}(nx):x\in\mathbb{R}_+)$. Hence, $q_a(n)$ must be the order of magnitude of $H^n_{b^{\ast}}(n)$. That is $q_a(n)$ is of the order of the time that the walk $X^{n^{-b^{\ast}}}$ needs to reach $n$.

We now give a heuristic argument to find $q_a(n)$ and $b^{\ast}$.
When $X^{n^{-b^{\ast}}}(t)$ has given $k$ steps, it has an order  $kn^{-b^{\ast}}$. So it takes about $n^{b^{\ast}+1}$ steps to be on site $n$.
 We can think that the number of visits to each site $x$ is evenly
 distributed. Then each site is visited about $n^{b^{\ast}}$ times before $X^{n^{-b^{\ast}}}$ hits $n$.
 The time that the walks needs to reach $n$ is of the order of the time
spent in the largest trap. Thus we can estimate the total time spent by the
walk as the depth of the deepest trap (which is of order $n^{-1/\alpha}$)
multiplied by the number of visits to that trap. This gives a time of order $n^{1/\alpha+b^{\ast}}$.
What the previous arguments show is that $X^{n^{-b^{\ast}}}(t)$ arrives at $n$ at time $t\approx n^{1/\alpha+b^{\ast}}$ ($q_a(n)\approx n^{1/\alpha +b^{\ast}}$).
 But at that time we want to analyze a walk of drift $(n^{1/\alpha+b^{\ast}})^{-a}$.
 That is, we need that
 $a(1/\alpha+b^{\ast})=b^{\ast}.$
 In this way we find that $b^\ast:=a/[(1-a)\alpha].$

\subsection{The embedded discrete time walk}
 For each natural $n$, the \textit{clock processes} $S^n$ is defined as $S^n(0):=0$.
Furthermore $S^n(k)$ is the time of the $k$-th jump of $X^{n^{-b^\ast}}$. $S^n$ is extended to all $\mathbb{R}^+$ by setting  $S^n(s):=S^n(\lfloor s\rfloor)$.
To each drifted walk $X^{n^{b^\ast}}(t)$ we associate its corresponding \textit{embedded discrete time random walk} $(Y_i^{n^{-b^\ast}}: i \in \mathbb{N})$
 defined as $Y_i^{n^{-b^\ast}}:=X^{n^{-b^\ast}}(t)$ where $t$ satisfies: $S^n(i)\leq t<S^n(i+1)$.

Obviously $Y_i^{n^{-b^\ast}}$ is a discrete time random walk with drift $n^{-b^\ast}$. We can write
$$S(k)=\sum_{i=0}^{k-1}\tau_{Y_i} e_i,$$
where $(e_i)_{i\geq0}$ is an i.i.d. sequence of exponentially distributed random variables with mean $1$.

Define\\
$\epsilon=\epsilon(n):= n^{-b^\ast} $\\
$p=p(n):=(1+\epsilon (n))/2$ \\
$q=q(n):=(1-\epsilon (n))/2$ and\\
$\nu (n):= \lfloor c\log(n)n^{b^{\ast}}\rfloor$     with  $c>2$.

Let $\Xi(x,k)=\Xi (x,k,n)$ be the probability that $Y_i^{\epsilon(n)}$ hits $x$ before $k$ starting at $x+1$. Then we have that
$\Xi(x,k)= q+p\Xi(x+1,k)\Xi(x,k)$ and that $\Xi(k-2,k)=q.$
These observations give a difference equation and an initial
condition to compute $\Xi (x,k)$. Then we get that
\begin{equation}\label{psi}
\Xi (x,k)=r\frac{1-r^{k-x-1}}{1-r^{k-x}},
\end{equation}
where $r=r(n):=q(n)/p(n)$.
Using that formula we can see that the probability that the walk $Y_i^{\epsilon(n)}$ ever hits $x-1$ starting at $x$ is $r$.
We now present a backtracking estimate.
\begin{lem}\label{backtracking} Let $\cal {A}(n):=\{ \min_{i\leq j\leq \zeta_n(n)}(Y_j^{\epsilon(n)}-Y_i^{\epsilon(n)})\geq -\nu(n)\}$
where $\zeta_n(i):=\min\{ k\geq 0 :Y_k^{\epsilon(n)}=i\}$, then
$ \lim_{n \rightarrow \infty}\mathbb P [\cal A(n)]=1.$
\end {lem}
\begin{dem}
We can write
 $$\cal A^c(n)=\bigcup_{x=0}^{n-1}\left\{\min_{\zeta_n(x)\leq i\leq \zeta_n(n)}(Y_i^{\epsilon(n)}-Y_{\zeta_n(x)}^{\epsilon(n)})< -\nu(n)\right\}.$$

Hence
$$\cal A^c(n)\subseteq\bigcup_{x=0}^{n-1}\left\{\min_{\zeta_n(x)\leq i}(Y_i^{\epsilon(n)}-Y_{\zeta_n(x)}^{\epsilon(n)})< -\nu(n)\right\}.$$

But, in order to arrive from $x$ to $x-\nu(n)$,
for each $j=x-1,\ldots,x-\nu(n)$, starting from $j+1$ the
random walk $Y_i^{\epsilon(n)}$ needs to
hit $j$ in a finite time.
 Hence, it takes
$\nu(n)$ realizations of independent events (strong Markov property) of
 probability $r(n)$.  In other words
 $\mathbb P[\cal A^c(n)]\leq nr(n)^{\nu(n)}= n(1-\frac{2}{1+n^{b^\ast}})^{\nu(n)}$,
which can be bounded by $ n(1-\frac{1}{n^{b^\ast}})^{\nu(n)}.$
Replacing $\nu(n)$ we obtain $n((1-\frac{1}{n^{b^\ast}})^{n^{b^\ast}})^{c\log(n)}.$
We can see that $(1-\frac{1}{n^{b^\ast}})^{n^{b^\ast}} \rightarrow e^{-1} \text
{ when } n\rightarrow \infty .$
Now, for $n$ big enough
$\left(1-\frac{1}{n^{b^\ast}}\right)^{n^{b^\ast}}\leq e^{-\frac{1}{2}}$. Then
$$\mathbb P[\cal A^c(n)]\leq nn^{-\frac{1}{2} c}.$$
But $c > 2$, so we get the result.
\end{dem}

Now we state the convergence result for the hitting time processes.
\begin{lem}\label{maximum}
Let
\begin{equation}\label{hittigntimenormalizad}
 H^{(n)}(t):= \frac{H_{b^{\ast}}^{n}(tn)}{n^{(1/\alpha)+b^{\ast}}}.
 \end{equation}
  Then $(H^{(n)}(t); t\in [0,T])$ converges weakly to
$((\frac{\pi\alpha}{\sin(\pi\alpha)})^{-1/\alpha}V_\alpha(t); t\in[0,T])$ on $(D[0,T],M_1)$, where $V_\alpha(t)$ is an $\alpha$-stable subordinator.
 \end{lem}

The proof of this lemma will be given in subsection \ref{proofmax}.
We present the proof of part $(i)$ of Theorem \ref{main} using lemma \ref{maximum} and devote the rest of the section to the proof of lemma \ref{maximum}.
\subsection{Proof of (i) of Theorem \ref{main}}
Let us denote $$\bar{X}^n(t):=n^{-1}\max \{X^{n^{-b^\ast}}(sn^{(1/\alpha)+b^{\ast}}); s\in[0,t]\}.$$

First we will prove convergence in distribution of $\bar{X}^n$ to the (right continuous generalized) inverse of $(\frac{\pi\alpha}{\sin(\pi\alpha)})^{-1/\alpha}V_\alpha$
in the uniform topology.
That is, we want to to prove convergence in distribution of the inverse of $(H^{(n)}(t); t\in [0,T])$  to the inverse of
 $((\frac{\pi\alpha}{\sin(\pi\alpha)})^{-1/\alpha}V_\alpha(t); t\in[0,T])$ in the uniform topology.
Define $$\cal C (T,S) _n:=\{H^{(n)}(S) \geq Tn^{(1/\alpha)+b^{\ast}}\}.$$

Then, we have that, on $\cal{C} (T,S)_n$, the right continuous generalized inverse of $(H^{(n)}(s);s\in[0,S])$ is $(\bar{X}^{n}(t);t\in[0,T])$.
Let $T>0$ be fixed, by Lemma \ref{maximum}, we know that we can choose $S$ big enough
so that $\lim_{n\to\infty}\mathbb{P}[\cal C (T,S)_n]$ is as close to $1$ as we want.
Let $D^{\uparrow}[0,T]$ be the subset of $D[0,T]$ consisting of the increasing functions.
By corollary 13.6.4 of \cite{witt}, the inversion map from $(D^{\uparrow}[0,T],M_1)$ to $(D^{\uparrow}[0,T],U)$ is continuous at strictly increasing functions.
Lemma (\ref{maximum}) gives convergence in distribution of $(H^{(n)}(t); t\in [0,S])$ to $((\frac{\pi\alpha}{\sin(\pi\alpha)})^{-1/\alpha}V_\alpha(t); t\in[0,S])$
 in the Skorohod $M_1$ topology. We know that $V_{\alpha}$ is a. s. strictly increasing, that is
 $((\frac{\pi\alpha}{\sin(\pi\alpha)})^{-1/\alpha}V_\alpha(t); t\in[0,S])\in D^\uparrow[0,T]$ almost surely.
 So we can apply corollary 13.6.4 of \cite{witt} and deduce convergence in distribution of $\bar{X}^n$ to the inverse of
 $(\frac{\pi\alpha}{\sin(\pi\alpha)})^{-1/\alpha}V_\alpha$ in the uniform topology.
 As we have said previously, the inverse of $(H^{(n)}(s);s\in[0,S])$ is $(\bar{X}^{n}(t);t\in[0,T])$  in $\cal{C} (T,S)_n$. This proves
convergence of the maximum of the walk. To deduce convergence of the walk itself it suffices to show that the walk is close enough to its maximum
in the uniform topology.
 That is, to prove the theorem, it is enough to show that for all $\gamma > 0$:
$$ \mathbb P\left[\sup_{0 \leq t \leq T}|n^{-1}X^{n^{-b^\ast}}(tn^{(1/\alpha)+b^{\ast}})-\bar{X}^{n}(t)| \geq \gamma\right]\rightarrow 0 .$$

 Again, by Lemma \ref{maximum} we know that $\mathbb P[H^n_{b^{\ast}}(n\log(n))\geq Tn^{(1/\alpha)+b^{\ast}}]\rightarrow 1$.  Hence, we just have to prove that
 $$\mathbb{P}\left[\sup_{0 \leq t \leq H^n_{b^{\ast}}(n\log(n))} |n^{-1}X^{n^{-b^\ast}}(t)-n^{-1}\max \{X^{n^{-b^\ast}}(s); s\in[0,t]\}|\geq \gamma\right]\rightarrow 0.$$
Which is to say,
$$ \mathbb{P}\left[\sup_{0 \leq k\leq \zeta_n(\lfloor n\log(n) \rfloor)}|Y^{\epsilon(n)}_k-\bar{Y}^{\epsilon(n)}_k| \geq
 n \gamma \right]\rightarrow 0 .$$
where $\bar{Y}^{\epsilon(n)}$ is the maximum of $Y^{\epsilon(n)}$.
But, we can apply Lemma \ref{backtracking} to see that this is the case.
\subsection {The environment}
Here we give estimates concerning the environment.
For each $n\in\mathbb N$ define
 $$g(n):=\frac{n^{1/\alpha}}{(\log(n))^\frac{2}{1-\alpha}}.$$

Now, for each site
 $x\in \mathbb N$, we say that
 $x$ is an {\it $n$-deep trap} if $\tau_x \geq g(n)$.
Otherwise we will say that $x$ is an {\it $n$-shallow trap}.
We now order the set of $n$-deep traps according to their position
from left to right. Then call $\delta_1(n)$  the leftmost $n$-deep trap
and in general call for $j\ge 1$, $\delta_j(n)$  the $j$-th $n$-deep trap. The number of $n$-deep traps in $[0,n]$ is denoted by $\theta (n)$.
Let us now define
$$\cal E_1 (n) := \left\{n\varphi(n)\left(1-\frac{1}{\log(n)}\right)\leq \theta_n \leq n\varphi(n)\left(1+\frac{1}{\log(n)}\right)\right\},$$
$$\cal E_2 (n):=\{\delta_1\wedge(\min_{1\leq j\leq \theta_n-1}(\delta_j-\delta_{j-1}))\leq\rho(n)\} ,$$
$$\cal E_3 (n):=\{\max_{-\nu(n)\leq x\leq 0} \tau_x <g(n) \}, \textrm{                  and}$$
$$\cal E(n):=\cal E_1 (n)\cap\cal E_2 (n)\cap\cal E_3 (n)$$
where $\rho(n):= n^\kappa \ \ \ \kappa<1$ and
$\varphi (n):= \mathbb P[\tau_x\geq g(n)] $.
 \begin{lem} We have that
$ \lim_{n\rightarrow\infty} \mathbb P [\cal E (n)]=1.$
\end {lem}
\begin {dem} $\theta (n)$ is binomial with parameters $(n,\varphi(n))$. $\cal E_1$ is estimated using the Markov inequality. To
control $\cal E_2$ it is enough to see that in ${0,..,n}$ there are $O(n\rho(n))$
pairs of points at a distance less than $\rho (n)$. The estimate on $\cal E_3$ is trivial.
\end {dem}

\subsection {Time control}
In this subsection we prove results about the time spent by the walk on the traps.
\subsubsection {Shallow traps} Here we will show that the time that the walks spend in the shallow traps is negligible.
 \begin{lem}\label{shallow}
 Let $\cal I(n):=\left\{\sum_{i=0}^{\zeta_n(n)} \tau_{Y^{\epsilon(n)}_i} \textbf {$e_i 1$}_{\{\tau_{Y^{\epsilon(n)}_i}\leq\ g(n)\} } \leq \frac{n^{1/[(1-a)\alpha]}}{\log (n)} \right\}
 $. Then
 \begin{equation} \label{tpp}
 \mathbb P[\cal I(n)]\rightarrow 1 \textrm{           as                } n\rightarrow\infty .
 \end{equation}
 \end{lem}
 \begin{dem}
We have that
$\mathbb P[\cal I(n)^c]=\mathbb P[\cal I(n)^c\cap\cal E(n)]+ o(1) $.
Using the Markov inequality it suffices to show that
$$\mathbb E \left[\sum_{i=0}^{\zeta_n(n)}\tau_{Y^{\epsilon(n)}_i}e_i1_{\{\tau_{Y^{\epsilon(n)}_i}<g(n)\}}1_{\{Y^{\epsilon(n)}_i\geq-\nu(n)\}}\right]
=o\left(\frac{n^{1/[(1-a)\alpha]}}{\log(n)}\right). $$

The number of visits of $Y^{\epsilon(n)}_i$ to $x$ before time $\zeta_n(n)$ is $1+G(x,n)$, where $G(x,n)$ is
a geometrically distributed random variable of parameter $1-(q+p\Xi(x,n))$ (the parameter is the probability that, $Y^{\epsilon(n)}_i$, starting at $x$, hits
$n$ before returning to $x$). Also
\begin{equation}\label{trampaschicas}
\mathbb E_\tau\left[\sum_{i=0}^{\zeta_n(n)}\tau_{Y^{\epsilon(n)}_i}e_i1_{\{\tau_{Y^{\epsilon(n)}_i}<g(n)\}}1_{\{Y^{\epsilon(n)}_i\geq-\nu(n)\}}\right]
\leq\sum_{x=-\nu(n)}^n\tau_x(1+\mathbb E_\tau[G(x,n)])1_{\{\tau_x<g(n)\}}.
\end{equation}

Using (\ref{psi}) we can deduce that $(1+\mathbb E[G(x,n)])\leq\frac{(1-r(n))}{p}\leq cn^{-b^\ast}$.
So, averaging with respect to the environment in (\ref{trampaschicas}) we get
$$\mathbb E \left[\sum_{i=0}^{\zeta_n(n)}\tau_{Y^{\epsilon(n)}_i}e_i1_{\{\tau_{Y^{\epsilon(n)}_i}<g(n)\}}1_{\{Y^{\epsilon(n)}_i\geq-\nu(n)\}}\right]\leq C n^{1+b^\ast}\mathbb
E[\tau_0 1_{\{\tau_0<g(n)\}}].$$
Also
$$\mathbb E[\tau_0 1_{\{\tau_0<g(n)\}}]\leq \sum_{j=0}^{\infty} (1/2)^jg(n)\mathbb P[\tau_0>(1/2)^{j+1}g(n)].$$

Now, using (\ref{heavytails}) there exists  a constant $C$ such that the righthand side of the above inequality is bounded above by
$$ C g(n)^{1-\alpha} \sum_{j=0}^{\infty} ((1/2)^{1-\alpha})^j.$$

Furthermore, since $1-\alpha > 0$ this expression is bounded above by $C g(n)^{1-\alpha}.$
This finishes the proof. \end{dem}
 \subsubsection{Deep traps} Here we will estimate the time spent in deep traps.
We define the occupation time for $x \in \mathbb Z$ as
$$T_x=T_x(n):= \sum_{i=0}^{\zeta_{n}(n)}\tau_{Y^{\epsilon(n)}_i}\textbf{$e_i$ 1}_{\{Y^{\epsilon(n)}_i=x\}}. $$

 The walk visits $x$, $G(x,n)+1$ times before $\zeta_n(n)$, and each visit lasts an exponentially distributed time.
This allows us to control the Laplace transform of $T_x$. For any pair of sequences of real numbers
$(a_n)_{n\in\mathbb{N}}$, $(b_n)_{n\in\mathbb{N}}$, $a_n\sim b_n$ will mean that $\lim_{n\rightarrow\infty} \frac{a_n}{b_n}=1$.
 \begin{lem}\label{deeptraps}
Let $\lambda > 0$. Define $\lambda_n :=
\frac{\lambda}{n^{1/[(1-a)\alpha]}}$. Then we have that
$$ \mathbb E^x[1-\exp(-\lambda_nT_x)|\tau_x\geq g(n)]\sim\frac{\mathbb P[\tau_x\geq g(n)]^{-1}\alpha\pi\lambda^{-\alpha}}{n\sin(\alpha \pi)}.$$
 \end{lem}
 \begin{dem}
   We must perform an auxiliary computation about the asymptotic behavior of the parameter $1-(q+p\Xi(x,n))$ of $G(x,n)$:

   $$(1-(q+p\Xi(x,n)))n^{b^{\ast}}=p\frac{1-r}{1-r^{n-x}}$$
  $$=\frac{2p(1+n^{-b^{\ast}})^{n-x}}{(1+n^{-b^{\ast}})((1+n^{-b^{\ast}})^{n-x}-(1-n^{-b^{\ast}})^{n-x})}$$
 $$=\frac{2p}{(1+n^{-b^{\ast}})(1-(1-\frac{2n^{-b^{\ast}}}{1+n^{-b^{\ast}}})^{n-x})}$$
which converges to $1$. Thus we have showed that
\begin{equation}\label{parametro}
 (1-(q+p\Xi(x,n)))n^{b^{\ast}}\stackrel{n\to\infty}{\to}1
\end{equation}

We have
  $$ \mathbb E^x_\tau[\exp(-\lambda_nT_x)]=\mathbb E^x_\tau\left[\exp\left(-\lambda_n\sum_{i=0}^{G(x,n)}\tau_x \tilde{e}_i\right)\right]$$
 where  $\tilde{e}_i $ are i.i.d. end exponentially distributed with
 $\mathbb{E}(\tilde{e}_i)=1$. Let $\tilde{\lambda}_n:=\frac{\lambda}{n^{1/\alpha}}$. Then
  $$ \mathbb E^x_\tau[\exp(-\lambda_nT_x)]=\frac{1}{1+\tilde{\lambda}_n\frac{\tau_x}{n^{b^\ast}(1-(q+p\Xi(x,n)))}}.$$
Using (\ref{parametro}) we get that the above expression equals
  $$=\frac{1}{1+\tilde{\lambda}_n\tau_x(1+o(1))}=\frac{1}{1+\tilde{\lambda}_n\tau_x}+o(n^{-1/\alpha}).$$
   Averaging with respect to the environment
  $$\mathbb E^x[1-\exp(-\lambda_nT_x)1_{\{\tau_x\geq g(n)\}}]=\int_{g(n)}^\infty 1-\frac{1}{1+{\tilde{\lambda}}_nz} \tau_0(dz) + o(n^{-1/\alpha})$$
  where the notation $\tau_0(dz)$ denotes integration with respect the distribution of $\tau_0$.
  Integrating by parts $\int_{g(n)}^\infty 1-\frac{1}{1+{\tilde{\lambda}}_nz} \tau_0(dz)$ we get that the above display equals
  $$\left[-\frac{\tilde{\lambda}_nz}{1+\tilde{\lambda_n}}\mathbb P[\tau_0\geq z]\right]_{g(n)}^{\infty}+\int_{g(n)}^{\infty}\frac{\tilde{\lambda}_n}{(1+\tilde{\lambda}_nz)^2}
\mathbb P[\tau_0\geq z] dz + o(n^{-1/\alpha}).$$

The first term is smaller than $C\tilde{\lambda}_n g(n)^{1-\alpha}=o(n^{-1})$. To estimate the second term, note that for all $\eta > 0$ we have
$$ (1-\eta)z^{-\alpha}\leq \mathbb P[\tau_0\geq z]\leq
(1+\eta)z^{-\alpha}$$
 for $z$ large enough. Then we must compute
$\int_{g(n)}^{\infty}\frac{\tilde{\lambda}_n}{(1+\tilde{\lambda}_nz)^2}z^{-\alpha}dz.$
Changing variables with  $y=\frac{\tilde{\lambda}_n z}{1+\tilde{\lambda}_n z}$ we obtain
$$\tilde{\lambda}_n^{-\alpha}\int_\frac{\tilde{\lambda}_n
g(n)}{1+\tilde{\lambda}_n g(n)}^1 y^{-\alpha}(1-y)^\alpha dy.$$
But we know that this integral converges to $\Gamma(\alpha+1)\Gamma(\alpha-1)=\frac{\pi\alpha}{\sin(\pi\alpha)}$.
 \end{dem}
 \subsection{Proof of Lemma \ref{maximum}}\label{proofmax}

We will show the convergence of the finite dimensional Laplace
transforms of the rescaled hitting times to the corresponding expression for an $\alpha$-stable
subordinator. This will prove finite dimensional convergence.

Let $0=u_0<\cdots<u_K\leq T$ and $\beta_i, i=1..K$ be positive numbers. We know that
$$\mathbb{E}\left[\exp\sum_{i=1}^{K}-\beta_i((\frac{\pi\alpha}{\sin(\pi\alpha)})^{-1/\alpha}V_\alpha(u_i)-(\frac{\pi\alpha}{\sin(\pi\alpha)})^{-1/\alpha}V_\alpha(u_{i-1}))\right]
=\exp\left(\sum_{i=1}^K -\frac{\alpha\pi\beta_K^{-\alpha}}{\sin(\alpha\pi)}(u_K-u_{K-1})\right).$$
So, it only suffices to show that
$$\mathbb{E}\left[\exp\sum_{i=1}^{K}-\beta_i(H^{(n)}(u_i)-H^{(n)}(u_{i-1}))\right]\stackrel{n\to\infty}{\to}\exp\left(\sum_{i=1}^K -\frac{\alpha\pi\beta_K^{-\alpha}}{\sin(\alpha\pi)}(u_K-u_{K-1})\right)$$
where $H^{(n)}$ is as in (\ref{hittigntimenormalizad}).
We can decompose the trajectory of $Y^{\epsilon(n)}$ up to $\zeta_n(\lfloor n u_K \rfloor)$ into three parts.
 The first one is the trajectory up to the time $\zeta_n(\lfloor n u_{K-1}-\nu (Tn) \rfloor)$,
 the second one is the
trajectory between times $\zeta_n(\lfloor n u_{K-1}-\nu (Tn)\rfloor)$ and
$\zeta_n(\lfloor n u_{K-1}\rfloor)$, finally, the third part
is the trajectory starting  from time $\zeta_n(\lfloor nu_{K-1}\rfloor)$
up to time $\zeta_n(\lfloor nu_{K}\rfloor)$.
First we will show that the time spent in the second part of the trajectory is negligible.
We have that $\mathbb P[\max_{y\in B_{\nu(Tn)}(x)}>g(Tn)]=o(1)$, which is to say that the probability of finding
 an $n$-deep trap in a ball of radius $\nu(Tn)$ is small. Indeed Lemma \ref{shallow} implies that there exists a constant $C>0$ such that
$$\mathbb P\left[\sum_{i=0}^{\zeta_{\lfloor u_Kn\rfloor}}\tau_{Y^{\epsilon(n)}_i}e_i1_{\left\{\tau_{Y^{\epsilon(n)}_i}\in B_{\nu(Tn)}(\lfloor u_{K-1}n\rfloor)\right\}}
<Cn^{\frac{1}{(1-a)\alpha}}(\log(n))^{-1}\right]
\rightarrow1.$$

 Hence, the time that the walk spends in $B_{\nu(Tn)}(\lfloor u_{K-1}n\rfloor)$
is negligible. But in $\cal A(Tn)$ the  walk never backtracks a distance
larger than $\nu(Tn)$, so, the time spent in the second part of the decomposition is negligible.
The fact that in $\cal A(Tn)$ the  walk never backtracks a distance
larger than $\nu(Tn)$ also implies that, conditional on $\cal A(Tn)$, the first and the third parts of the decomposition of the trajectory
corresponds to independent walks in independent environments.

 So $\mathbb E[\exp(\sum_{i=1}^{K}-\beta_i(H^{(n)}(u_i)-H^{(n)}(u_{i-1})))]$ can be expressed as
 $$ \mathbb E\left[\exp\sum_{i=1}^{K-1}-\beta_i(H^{(n)}(u_i)-H^{(n)}(u_{i-1}))\right]\mathbb E^{\lfloor nu_{K-1}\rfloor}\left[\exp-\beta_K(H^{(n)}(u_i)-H^{(n)}(u_{i-1}))\right]+o(1) $$
 where $o(1)$ is taking account of the time spent in the second part of the decomposition of the trajectory and of $\cal A(Tn)^c$.

The strong Markov property of $Y^{\epsilon(n)}$ applied at the stopping time $\zeta_n(\lfloor nu_{K-1}\rfloor)$ and translational invariance of the environment give that $H^{n}_{b^{\ast}}(n u_i)-H^{n}_{b^{\ast}}(n u_{i-1})$ is distributed as $H^{n}_{b^{\ast}}(ns_n(K))$
  where $s_n(K)= \frac{\lfloor u_K n\rfloor-\lceil u_{K-1} n\rceil}{n}$.
Iterating this procedure $K-2$ times we reduce the problem to the computation of one-dimensional Laplace transforms.
Hence, we have to prove that, for each $k\leq K$
$$\mathbb E[ \exp(-\beta_k n^{-(1/\alpha)-a}H^n_{b^{\ast}}(ns_n(k)))]\rightarrow\exp\left(-\frac{\pi\alpha}{\sin(\pi\alpha)}\beta_k^\alpha(u_k-u_{k-1})\right).$$

We have that $\mathbb{P}[\cal E(Tn)\cap A(Tn)]\to1$, then we can write
$$\mathbb E[\exp(-\beta_kn^{-(1/\alpha)-a}H^n_{b^{\ast}}(ns_n(k)))]=
\mathbb E[\exp(-\beta_kn^{-(1/\alpha)-a}H^n_{b^{\ast}}(ns_n(k)))1_{\{\cal E(Tn)\cup A(Tn)\}}]+o(1).$$
We know that the time spent in the shallow traps is negligible, so we only have to take into account the deep traps.
We also know that on $A(Tn)$, the walk does not backtrack more than $\nu(Tn)$, and
that, on $\cal E(Tn)$, the deep traps on $[0,Tn]$ are well separated. Then we can write
$$\mathbb E[ \exp(-\beta_k n^{-(1/\alpha)-a}H^n_{b^{\ast}}(ns_n(k)))]=\mathbb E\left[\prod_{j=1}^{\theta(ns_n(k))}\mathbb E_{\tau}^{\delta_i}[\exp{(-\beta_kn^{-(1/\alpha)-a}T_{\delta_i})}]\right]+o(1).$$

Also, in  $\cal E(Tn)$ we have upper and lower bounds for $\theta (Tn)$. Using
the upper bound we see that the righthand side of the above equality
is bounded above by
$$ \mathbb E \left[\prod_{j=1}^{ns_n(k) \varphi (ns_n(k))(1-\frac{1}{\log(ns_n(k))})}\mathbb E_{\tau}^{\delta_i}[\exp{(-\beta_kn^{-(1/\alpha)-a}T_{\delta_i})}]\right]+o(1),$$

 Applying again the translational invariance of the environment and the strong Markov property we get that that the above display is equal to
 $$\mathbb{E}[\mathbb {E}_\tau^{\delta_i}[\exp{(-\beta_kn^{-(1/\alpha)-a}T_{\delta_i})}]]^{ns_n(k) \varphi (ns_n(k))(1-\frac{1}{\log(ns_n(k))})}+o(1)$$
which in turn can be expressed as
 $$\mathbb{E}[\exp{(-\beta_kn^{-(1/\alpha)-a}T_0)}|\tau_0\geq g(ns_n(k))]^{ns_n(k) \varphi (ns_n(k))(1-\frac{1}{\log(ns_n(k))})}+o(1).$$
 Using lemma (\ref{deeptraps}) and the fact that $s_n(k)\stackrel{n}{\to}u_k-u_{k-1}$ we obtain
 $$ \limsup\mathbb E[\exp(-\beta_kn^{-(1/\alpha)-a}H^n_{b^{\ast}}(ns_n(k)))] \leq \exp\left( -\frac{\alpha\pi\beta_k^{-\alpha}}{\sin(\alpha\pi)}(u_k-u_{k-1})\right) .$$

 The lower bound can be obtained in an analogous fashion.
For the tightness, the arguments are the same as in Chapter 5 of \cite{p-spin}
\section{\textbf{The Critical Case}}\label{critical}
 We want to show that for $a=\frac{\alpha}{\alpha+1}$ the sequence of walks $(X^{(n,a)}(t);t\in[0,\infty])$ converges in distribution to a drifted F.I.N. diffusion.
 We will mimic the arguments in \cite{fin}. But to treat the asymmetry of the model we will use a Brownian motion with drift instead of a Brownian motion. We use the existence of a bi-continuous version of the local time for a Brownian motion with drift.
 \subsection{The construction of the walks}
 Recall the definition of $\tilde{X}(\mu)$ given in
display (\ref{timechange2}).
Let $s$ be a real number and define
$$\mu:=\sum_{i\in \mathbb(Z)} v_i\delta_{si}.$$

Then $\tilde{X}(\mu)$ is a homogeneus Markov process with $s\mathbb{Z}$ as its state space.
 The transition probabilities and jump rates of $\tilde{X}(\mu)$ can be computed from the positions and weights of the atoms using the generator $L$ of $C(t)$
\begin{equation} \label{generator2}
Lf:=\frac{1}{2}\frac{d^2f}{d x^2}+\frac{df}{d x}.
\end{equation}

The arguments we will give below are an adaptation of the reasoning used by Stone in \cite{stone}.
For each $i$ let $\eta_{si}$ be the time of the first jump of $\tilde{X}(\mu)$ started at $si$.
By construction we will have that $\eta_{si}=v_i\tilde{l}(\sigma_s,0)$, where $\sigma_s$ is the hitting time of $(-s,s)$ by $C(t)$.
Using the strong Markov property for $C(t)$ we can deduce that $\eta_{si}$ is exponentially distributed.
It is easy to see that its mean is $v_i\mathbb{E}[\tilde l(\sigma_s,0)]$.
Denote by $p_t(x)$ the density at site $x$ of the distribution of $C(t)$ absorbed at $\{-s,s\}$.
Using that $\tilde{l}(\sigma_s,0):=\epsilon^{-1}\lim_{\epsilon\to0}m(t\in[0,\sigma_s]:C(t)\in[-\epsilon,\epsilon])$ and applying Fubini`s Theorem we find that
$\mathbb{E}[\tilde{l}(\sigma_s,0)]=\epsilon^{-1}\lim_{\epsilon\to0}\int_0^{\sigma_s}\mathbb{P}[C(t)\in[-\epsilon,\epsilon]]dt$.
 Then we find that
$$\mathbb{E}[\tilde{l}(\sigma_s,0)]=\int_0^{\infty}p_t(0)dt.$$

 We also know that $\int_0^{\infty}p_t(0)dt=f(0),$
where $f$ is the Green function of (\ref{generator2}) with Dirichlet conditions on $\{-s,s\}$.
That is, $f$ is the continuous function that satisfies
 $$\frac{1}{2}\frac{d^2 f}{d x^2}+\frac{d f}{d x}=-\delta_0 \textrm {           and              } f(s)=f(-s)=0.$$
We know that the general solution to $\frac{1}{2}\frac{d^2 g}{d x^2}+\frac{d g}{d x}=0$ is
$g=C_1\exp(-2x) + C_2.$
This and the constraints on $f$ give that
\begin{equation}\label{rates}
 \mathbb{E}[\eta_{si}]=v_i^{-1}\frac{\exp(-2s)+1}{1-\exp(-2s)}.
\end{equation}

For the computation of the respective transition probabilities we can use again the generator $L$.
 Let $g:[-s,s]\rightarrow\mathbb{R}$ be a continuous function such that
 $\frac{1}{2}\frac{d^2 g}{d x^2}+\frac{d g}{d x}=0$
and $g(-s)=0, g(s)=1$.
Using It\={o}'s formula, we find that that $g(C(t))$ is a martingale.
By the optional stopping theorem with the stopping time $\sigma_s$ we find that the probability that the walk takes his first step to the right is $g(0)$.
We can use the constraints on $g$ to see that
\begin{equation} \label{transitionprobabilities}
 \mathbb{P} [ \tilde{X}(\mu)(\eta_{si})=s(i+1)] = \frac{\exp(2s)}{1+\exp(2s)}.
\end{equation}

The proof of part $(ii)$ of Theorem (\ref{main}) will rely strongly on the
following proposition.
\begin{prop}\label{stonedrift}
Let $(\nu_n)_{n\in \mathbb N }$ be a sequence of measures that converges
vaguely to $\nu$, a measure whose support is $\mathbb{R}$.  Then the corresponding processes
$(\tilde{X}(\nu_n)(t),0\leq t\leq T)$ converges to $(\tilde{X}(\nu)(t),0\leq t\leq T)$ in distribution in $(D[0,T],U)$.
\end{prop}

 For the case where the underlying process is a Brownian motion, the proof of this fact can be found in \cite{stone}.
 We will use the continuity properties for the local time $\tilde{l}$. For each fixed $t$, $\tilde{l}$ is continuous and of compact support in $x$.
 Then, the vague convergence of $\nu_n$ to $\nu$ implies the almost sure convergence of $\tilde{\phi}_{\nu_n}(t)$ to $\tilde{\phi}_{\nu}(t)$.
 As $\tilde{l}$ is continuous in $t$, we obtain continuity of $\tilde{\phi}_{\nu_n}$ and of $\tilde{\phi}_{\nu}$.
 That, plus the fact that the $\tilde{\phi}_{\nu_n}$ are non-decreasing implies that that $\tilde{\phi}_{\nu_n}$ converges uniformly to $\tilde{\phi}_{\nu}$.
 The function $\tilde{\phi}_{\nu}$ is almost surely strictly increasing, because the support of $\nu$ is $\mathbb{R}$. Now we can apply corollary 13.6.4 of \cite{witt} to obtain that $\tilde{\psi}_{\nu_n}$ converges uniformly to $\tilde{\psi}_{\nu}$. That plus the continuity of the Brownian paths yields the lemma.
\subsection{The coupled walks}
To prove part (ii) of Theorem \ref{main}, we will use Proposition \ref{stonedrift}. That is we want to show that each walk $(X^{(n,a)}(t);t\in[0,\infty])$ can be expressed as
a speed measure change of $C(t)$, and then use convergence of the measures to get convergence of the
processes.
 The problem is that we are dealing with a sequence of random measures, and the proposition deals only with deterministic measures.
 To overcome this obstacle we can construct a coupled sequence of random measures $(\rho_n)_{n\in\mathbb{N}}$, such that
$(\tilde{X}(\rho_n)(t);t\in[0,\infty])$ is distributed as $(X^{(n,a)}(t);t\in[0,\infty])$ and that $(\rho_n)_{n\in\mathbb{N}}$ converges
almost surely vaguely to $\rho$, where $\rho$ is the random measure defined in (\ref{rho}) such that $\tilde{Z}=\tilde{X}[\rho]$. This section is devoted to the construction of the coupled measures.

We recall that $V_{\alpha }$ is an $\alpha$-stable subordinator.  
To make the construction  clearer, we will first suppose that $\tau_0$ is equidistributed with the positive $\alpha$-stable distribution $V_{\alpha}(1)$.
 Let us consider the strictly increasing process $(\tilde{V}_{\rho}(t); t\in \mathbb{R})$
given by $\tilde{V}_{\rho}(t):=\rho [0,t]$ if $t \geq 0$ and
$\tilde{V}_{\rho}(t):=-\rho[t,0)$ if $t<0$. It is a known fact from the theory of Levy processes that $\tilde{V}_{\rho}(t)$ is a two sided $\alpha$-stable subordinator.
 We now use this process
to construct the coupled sequence of random measures $(\rho_n)_{n\in\mathbb{N}}$ as
$$\rho_n:=\sum_i n^{-1/(1+\alpha)}\tau_i^n\delta_{s_ni},$$
 where
 $s_n:=\frac{1}{2}\log{\frac{n^{-a}+1}{1-n^{-a}}} $
 and
\begin{equation}\label{tauene}
\tau_i^n:=n^{1/(1+\alpha)}(\tilde{V}_{\rho}(n^{-\alpha/(1+\alpha)}(i+1))-\tilde{V}_{\rho}(n^{-\alpha/(1+\alpha)}i)).
\end{equation}

 Observe that $(\tau_i^n)_{i \in \mathbb Z}$ is an i.i.d. sequence distributed
 like $\tau_0$, so that using (\ref{transitionprobabilities}) and (\ref{rates}) we see that $\tilde{X}(\rho_n)$
 is a walk with drift $n^{-1/\alpha}$ taking values in $s_n\mathbb {Z}$. The
 latter means that $\tilde X(\rho_n)$ is distributed like  $s_n n^{\frac{\alpha}{1+\alpha}}X^{(n,a)}$.
 The key observation here is that the scaling factor $s_n$ satisfies
 \begin{equation}\label{sn}
   s_nn^{\alpha/(1+\alpha)}\rightarrow 1 \textrm{           as               } n\rightarrow \infty.
 \end{equation}

So, we just have to show that $\tilde X(\rho_n)$ converges to
$\tilde X(\rho)$, because (\ref{sn}) implies that if $\tilde X(\rho_n)$
converges to $\tilde X(\rho)$, so $s_nn^{\alpha/(1+\alpha)}\tilde{X}(\rho_n)$ does.
  With (\ref{sn}) in mind it is easy to prove that the sequence of measures
  $(\rho_n)$ converges almost surely vaguely to $\rho$. Suppose that $a<b$ are real numbers and that $V_\rho$ is continuous at $a$ and $b$, then
$$\rho_n((a,b])=V_{\rho}(n^{-\alpha/(1+\alpha)}\lfloor a/s_n\rfloor)-V_{\rho}(n^{-\alpha/(1+\alpha)}(\lfloor b/s_n\rfloor+1)).$$
But using (\ref{sn}) it is clear that $n^{-\alpha/(1+\alpha)}\lfloor a/s_n\rfloor\stackrel{n\to\infty}{\to}a$ and $n^{-\alpha/(1+\alpha)}\lfloor b/s_n\rfloor\stackrel{n\to\infty}{\to}b$. Then the continuity of $V_\rho$ at $a$ and $b$ implies that $\rho_n((a,b])\stackrel{n\to\infty}{\to}\rho(a,b]$, and we have proves the vague convergence of $\rho_n$ to $\rho$.

Suppose now that $\tau_0$  is not a positive $\alpha$-stable random
variable. Then, we can follow Section 3 of \cite{fin}.
There they construct constants $c_\epsilon$ and functions $g_{\epsilon}$ such that $\tau_i^{(\epsilon)}$ is distributed like $\tau_0$, where
\begin{equation}\label{tauepsilon}
\tau_i^{(\epsilon)}:=c_{\epsilon}^{-1}g_\epsilon(\tilde{V}_{\rho}(\epsilon(i+1)) -\tilde{V}_{\rho}(\epsilon i)).
\end{equation}
Lemma 3.1 of \cite{fin} says that
\begin{equation}\label{gepsilon}
 g_{\epsilon}(y)\rightarrow y \textrm{             as                } \epsilon \rightarrow 0.
\end{equation}
As $\tau_0$ satisfies (\ref{tail}) and using the construction of $c_{\epsilon}$ in Section 3 of \cite{fin}, we can deduce that
\begin{equation}\label{cepsilon}
 c_{\epsilon}\sim\epsilon^{1/\alpha}
\end{equation}

 Define
\begin{equation}\label{tauiene}
 \tau_i^n:=\tau_i^{(n^{(-\alpha/(1+\alpha))})}
\end{equation}
and again
$$\rho_n:=\sum_i n^{-1/(1+\alpha)}\tau_i^n\delta_{s_ni}.$$

Then, by definition (\ref{tauepsilon}), $\tilde{X}(\rho_n)$
 is a walk with drift $n^{-1/\alpha}$ taking values in $s_n\mathbb {Z}$.
Using (\ref{gepsilon}), (\ref{cepsilon}) and (\ref{sn}) we can see that $\mathbb{P}$-a.s.
$\rho_n \rightarrow \rho \textrm{                              vaguely}.$

\section{\textbf{The Subcritical Regime}}\label{sub}
We will prove that if $a >\alpha/(1+\alpha)$, then $(X^{(n,a)};t\in[0,\infty])$ converges to a F.I.N. diffusion. We obtain the same scaling limit that was obtained in \cite{fin} for a symmetric B.T.M. Nevertheless, here we have to deal with walks which are not symmetric, in contrast with the situation of \cite{fin}. For this purpose we express each rescaled walk as a time scale change of a Brownian motion. The scale change is necessary to treat the asymmetry of the walk. Then we show that the scale change can be neglected.
 We now proceed to define a time scale change of a Brownian motion.
 Let $\mu$ be a locally finite
discrete measure
$$\mu(dx):=\sum_{i\in\mathbb{Z}}w_i\delta_{y_i}(dx),$$
where $(y_i)_{i\in\mathbb Z}$ is an ordered
sequence of real numbers so that $y_i<y_j \text{    i.i.f.    } i<j$.

Let $S:\mathbb{R}\to\mathbb{R}\cup\{\infty,-\infty\}$ be a real valued, strictly increasing function, 
 $\mu$ will be the speed-measure and $S$  the scaling
function of the time scale change of Brownian motion. Define the scaled measure $(S\circ\mu)(dx)$ as
$$(S\circ\mu)(dx):=\sum_iw_i\delta_{S(y_i)}(dx).$$

Let
 $$\phi(\mu,S)(t):=\int_{\mathbb{R}}l(t,y)(S\circ\mu)(dy)$$
and $\psi(\mu,S)(s)$ be the right continuous generalized inverse of $\psi(\mu,S)$. Then, as shown in \cite{stone}
$$X(\mu,S)(t):=S^{-1}(X(S\circ\mu)(t)))$$
is a continuous time random walk with $\{y_i\}$ as its state space. The mean of the exponentially
distributed waiting time of $X(S\circ\mu)$ on $y_i$ is
\begin{equation}\label{meanofthetime}
2w_i\frac{(S(y_{i+1})-S(y_i))(S(y_{i})-S(y_{i-1}))}{S(y_{i+1})-S(y_{i-1})}
\end{equation}
and the transition probabilities to the right and to the left respectively
are
\begin{equation}\label{probabilidadesdetransicion}
 \frac{S(y_{i+1})-S(y_i)}{S(y_{i+1})-S(y_{i-1})}\text{                   and                     }\frac{S(y_i)-S(y_{i-1})}{S(y_{i+1})-S(y_{i-1})}.
\end{equation}
As in the previous section, we need to define a sequence of measures
$(\nu_n)_{n\in\mathbb{Z}}$ converging almost surely vaguely to $\rho$,
and which can be used to express the sequence of rescaled walks $X^{(n,a)}$.

 Let
$$\nu_n:=\sum_{i\in\mathbb{Z}}\frac{r_n+1}{2r_n^i}\tau_i^n\delta_{in^{\alpha/(\alpha+1)}},$$
where $\tau_i^n $ are defined in display (\ref{tauiene}), and
$r_n:=1-\frac{2n^{-a}}{1+n^{-a}}.$
 We will also use a sequence of scaling functions $S^n$ (which will converge
 to the identity mapping) given by
$$S^n(in^{\alpha/(\alpha+1)}):=\sum_{j=0}^{i-1}\frac{r^j}{n^{\alpha/(\alpha+1)}}.$$

We extend the domain of definition of $S^{n}$ to $\mathbb{R}$ by linear interpolation.
Then, by (\ref{meanofthetime}) and (\ref{probabilidadesdetransicion}), we have that $X(\nu_n,S^n)$ is distributed like $X^{(n,a)}$. We will use the following theorem proved by Stone in \cite{stone}.
\begin{prop}\label{stone}
Let $(\nu_n)_{n\in \mathbb N }$ be a sequence of measures that converges vaguely to $\nu$. Then the corresponding processes
$(X(\nu_n)(t),0\leq t\leq T)$ converges to $(X(\nu)(t),0\leq t\leq T)$ in distribution in $(D[0,T],J_1)$
\end{prop}
The proof of part (iii) of theorem \ref{main} will rely in the following lemma.
Let $id$ denote the identity mapping on $\mathbb{R}$, then we have that
\begin{lem}\label{sandrho}
$S^n(n^{-\alpha/(1+\alpha)}\lfloor n^{\alpha/(\alpha+1)}\cdot\rfloor)$ converges uniformly on compacts to $id$ and $\nu_n$ to converges almost surely vaguely to $\rho$.
\end{lem}
\begin{proof}
The convergence of the scaling functions is easily seen to be true under the assumption $a>\alpha/(\alpha+1)$ because
$$S^n(n^{-\alpha/(1+\alpha)}\lfloor n^{\alpha/(1+\alpha)}x\rfloor)=\sum_{j=0}^{\lfloor n^{\alpha/(1+\alpha)}x\rfloor}\frac{r_n^j}{n^{\alpha/(\alpha+1)}}$$
and
$$\frac{r_n^{\lfloor n^{\alpha/(1+\alpha)}x\rfloor}\lfloor n^{\alpha/(\alpha+1)}x\rfloor}{n^{\alpha/(1+\alpha)}}\leq\sum_{j=0}^{\lfloor n^{\alpha/(1+\alpha)}x\rfloor}\frac{r_n^j}{n^{\alpha/(\alpha+1)}}\leq\frac{\lfloor n^{\alpha/(\alpha+1)}x\rfloor}{n^{\alpha/(1+\alpha)}}.$$

Now we use the fact that
$$r_n^{\lfloor n^{\alpha/(1+\alpha)}x\rfloor}=\left(1-\frac{2n^{-a}}{1+n^{-a}}\right)^{\lfloor n^{\alpha/(1+\alpha)}x\rfloor}$$
converges to $1$, because $a>\alpha/(1+\alpha)$.

In a similar fashion it can be shown that the ``correcting factors" $\frac{r_n+1}{2r_n^i}$ in the definition of $\nu_n$ converge uniformly to $1$
in any bounded interval. Hence, we can show the convergence of $\nu_n$ to $\rho$ as in the previous section.

\end{proof}
 Lemma \ref{sandrho} implies the vague convergence of $(S^n\circ\nu_n)$ to $\rho$.
 Then, by proposition \ref{stone} we can deduce that $X(S^n\circ\nu_n)$ converges to $X(\rho)$.
 Let $T>0$, by lemma \ref{sandrho} we have that $S^{-1}$ also converges uniformly to the identity. Thus, using the precedent observations, we get that
 $(X(\mu,S)(t):0\leq t\leq T)$ converges to $(X(\rho)(t)0\geq t\geq T)$ in $D[0,T]$ with the Skorohod $J-1$ topology.
 We have proved  that $(X^{(n,a)}(t); t \in [0;T])$ converges in distribution to the F.I.N. diffusion $(Z(t); t \in [0,t]) $ on $(D[0,T],J_1)$.
 
 Thus, it remains to prove that the convergence takes place also in the
 uniform topology. Using the fact that the support of $\mathbb{\rho}$ is $\mathbb{R}$, we can show that $\phi(\rho,id)$ is strictly increasing. The almost sure vague convergence of $S\circ\nu_n$ to $\rho$ implies that, for all $t\geq 0$, $\phi(\nu_n,S^{n})(t)$ converges to $\phi(\rho,id)(t)$. As $l$ is continuous in $t$, we obtain continuity of $\phi(\nu_n,S^{n})$ and of $\phi(\rho,id)$. That, plus the fact that the $\phi(\nu_n,id)$ are non-decreasing implies that that $\phi(\nu_n,S^{n})$ converges uniformly to $\phi(\rho,id)$.
 The function $\phi(\rho,id)$ is almost surely strictly increasing, because the support of $\rho$ is $\mathbb{R}$. Now we can apply corollary 13.6.4 of \cite{witt} to obtain that $\psi(\nu_n,S^n)$ converges uniformly to $\psi(\rho,id)$. That, plus the continuity of the Brownian paths yields that $X(S^n\circ \nu_n)$ converges uniformly to $X(\rho,id)$. Using that $S^{n -1}$ converges to the identity,  we finally get that $X(\nu_n,S^{n})$ converges uniformly to $X(\rho)$.

\bibliographystyle{amsplain}
\bibliography{manolo}

\providecommand{\bysame}{\leavevmode\hbox to3em{\hrulefill}\thinspace}
\providecommand{\MR}{\relax\ifhmode\unskip\space\fi MR }
\providecommand{\MRhref}[2]{%
  \href{http://www.ams.org/mathscinet-getitem?mr=#1}{#2}
}
\providecommand{\href}[2]{#2}
\begin{thebibliography}{10}

\bibitem{bbg}
G.~{Ben Arous}, A.~Bovier and V.~Gayrard, \emph{Glauber dynamics of the random
  energy model. {I}. {M}etastable motion on the extreme states}, Comm. Math.
  Phys. \textbf{235} (2003), no.~3, 379--425.

\bibitem{bbg2}
\bysame, \emph{Glauber dynamics of the random energy model. {II}. {A}ging below
  the critical temperature}, Comm. Math. Phys. \textbf{236} (2003), no.~1,
  1--54.

\bibitem{p-spin}
G.~{Ben Arous}, A.~Bovier and J.~\v{C}ern{\'y}, \emph{Universality of the
  {REM} for dynamics of mean-field spin glasses}, Comm. Math. Phys.
  \textbf{282} (2008), no.~3, 663--695.

\bibitem{bc}
G.~{Ben Arous} and J.~\v{C}ern{\'y}, \emph{Bouchaud's model exhibits two
  different aging regimes in dimension one}, Ann. Appl. Probab. \textbf{15}
  (2005), no.~2, 1161--1192.

\bibitem{bcnotes}
\bysame, \emph{Dynamics of trap models}, Mathematical statistical physics,
  Elsevier B. V., Amsterdam, 2006, pp.~331--394.

\bibitem{zetade}
\bysame, \emph{Scaling limit for trap models on {$\Bbb Z^d$}}, Ann. Probab.
  \textbf{35} (2007), no.~6, 2356--2384.

\bibitem{universal}
\bysame, \emph{The arcsine law as a universal aging scheme for trap models},
  Comm. Pure Appl. Math. \textbf{61} (2008), no.~3, 289--329.

\bibitem{two}
G.~{Ben Arous}, J.~\v{C}ern{\'y} and T.~Mountford, \emph{Aging in
  two-dimensional {B}ouchaud's model}, Probab. Theory Related Fields
  \textbf{134} (2006), no.~1, 1--43.

\bibitem{weakergodicitybreaking}
J.-P. Bouchaud, \emph{Weak ergodicity breaking and aging in disordered
  systems}, J. Phys.I (France) \textbf{2} (1992), 1705--1713.

\bibitem{phys}
L.~Cugliandolo, J.~Kurchan, J.-P. Bouchaud and M.~Mezard, \emph{Out of
  equilibrium dynamics in spin-glasses and other glassy systems}, World
  Scientific, Singapore, (1998).


\bibitem{fin}
L.~R.~G. Fontes, M.~Isopi and C.~M. Newman, \emph{Random walks with strongly
  inhomogeneous rates and singular diffusions: convergence, localization and
  aging in one dimension}, Ann. Probab. \textbf{30} (2002), no.~2, 579--604.


\bibitem{gmw}
P.~M{\"o}rters, N. Gantert and V.~Wachtel, \emph{Trap models with vanishing
  drift: Scaling limits and ageing regimes}, (2010), \textbf{arXiv} 1003.1490v1.

\bibitem{stone}
C.~Stone, \emph{Limit theorems for random walks, birth and death processes, and
  diffusion processes}, Illinois J. Math. \textbf{7} (1963), 638--660.

\bibitem{witt}
W.~Witt, \emph{Stochastic-process limits}, Springer series in operations
  research, Springer-Verlag New York, Inc., (2002).

\bibitem{zindy}
O.~Zindy, \emph{Scaling limit and aging for directed trap models}, Markov
  Process. Related Fields \textbf{15} (2009), no.~1, 31--50.

\end{thebibliography}

\end{document}